\documentclass[12pt,reqno]{amsart}

\usepackage{amssymb}
\usepackage{amsmath}
\usepackage{latexsym}
\usepackage{amsthm}
\usepackage{epsfig}
\usepackage{color}
\usepackage{comment}
\usepackage{amsfonts,amssymb,mathrsfs,amscd}

\usepackage{enumitem}

\theoremstyle{plain}
\newtheorem{theorem}{Theorem}[section]

\newtheorem{remark}{\bf Remark}[section]
\theoremstyle{definition}

\newcommand{\rot}{\mathop\mathrm{rot}}
\newcommand{\grad}{\mathop\mathrm{grad}}
\renewcommand{\div}{\mathop\mathrm{div}}
\newcommand{\Tr}{\mathop\mathrm{Tr}}

\title
[Lieb--Thirring inequalities on the sphere ]
 {Lieb--Thirring inequalities on the sphere}

\author[A.Ilyin and A.Laptev] {Alexei  Ilyin and Ari Laptev}

\begin{document}
 \begin{abstract}
 We prove on the sphere $\mathbb{S}^2$ the Lieb--Thirring
inequalities for orthonormal families of  scalar and vector functions both on the whole sphere and
on proper domains on  $\mathbb{S}^2$. By way of applications we
obtain an explicit estimate for the dimension of the attractor
of the Navier--Stokes system on a domain on the sphere with Dirichlet
non-slip boundary conditions.
\end{abstract}

\subjclass[2010]{35P15, 26D10, 35Q30.}
\keywords{Lieb--Thirring inequalities, Spectral inequalities on the sphere,
Navier--Stokes equations, Attractors}

\address
{\noindent\newline  Keldysh Institute of Applied Mathematics;
\newline
Imperial College London and Institute Mittag--Leffler;}
\email{ ilyin@keldysh.ru;  a.laptev@imperial.ac.uk}

\maketitle

\setcounter{equation}{0}
\section{Introduction}\label{sec1}
\label{sec0}

The Schr\"odinger operator in $L_2(\mathbb{R}^d)$
$$
-\Delta-V
$$
with a real-valued  potential $V$ sufficiently fast decaying at infinity
has a discrete  negative
spectrum satisfying the Lieb--Thirring
spectral inequalities~\cite{LT}, \cite{lthbook}, \cite{L-W}
\begin{equation}\label{LT}
\sum_{\lambda_i\le0}|\lambda_i|^\gamma\le\mathrm{L}_{\gamma,d}
 \int_{\mathbb{R}^d} V_+(x)^{\gamma+ d/2}dx,
\end{equation}
where $V_+(x)=(|V(x)|+ V(x))/2$ is the positive part of $V$.

For $\gamma\ge 3/2$ the sharp value of the Lieb--Thirring constants  $\mathrm{L}_{\gamma,d}$
was found in~\cite{Lap-Weid}:
\begin{equation}\label{Acta}
\mathrm{L}_{\gamma,d}= \mathrm{L}_{\gamma,d}^{\mathrm{cl}}:=
\frac1{(2\pi)^d}\int_{\mathbb{R}^d}(1-|\xi|^2)_+^\gamma dx=
\frac{\Gamma(\gamma+1)}
{(4\pi)^{d/2}\Gamma(d/2+\gamma+1)}\,.
\end{equation}

For $1\le\gamma<3/2$ the best known estimate of $\mathrm{L}_{\gamma,d}$
was obtained in \cite{D-L-L}:
\begin{equation}\label{DLL}
\mathrm{L}_{\gamma,d}\le \frac\pi{\sqrt{3}}
\cdot\mathrm{L}_{\gamma,n}^{\mathrm{cl}},
\quad \frac\pi{\sqrt{3}}=1.8138\dots\,.
\end{equation}

The proof of \eqref{Acta} is essentially based on the following two
ingredients, while the third one, in addition, is essential for the proof of~\eqref{DLL}:
\begin{itemize}
  \item[] 1). Sharp one-dimensional  Lieb--Thirring estimates
  for  the Scrh\"odin\-ger operators on $\mathbb{R}$ with matrix-valued potentials for $\gamma\ge3/2$~\cite{Lap-Weid}.
  \item[] 2). The Cartesian coordinates in $\mathbb{R}^d$, which make it possible to
  use the  lifting argument with respect to dimensions \cite{LapFA}, \cite{Lap-Weid}.
  \item[] 3). One-dimensional  Lieb--Thirring estimates
  for  the Scrh\"odinger operators on $\mathbb{R}$ with matrix-valued potentials for
$\gamma=1$, or, equivalently, one-dimensional generalized Sobolev inequality for traces of matrices~\cite{D-L-L}, \cite{E-F}.
\end{itemize}

We are interested in this work in Lieb--Thirring inequalities on the sphere.
Lieb--Thirring inequalities were generalized in \cite{G-M-T}, \cite{T}
to higher-order operators and Riemannian manifolds, however, the method
involves extension operators for Sobolev spaces and therefore no information
is available on the corresponding constants.

Explicit bounds for the Lieb--Thirring constants on the sphere
were found in \cite{ILMS} by means of the Birman--Schwinger kernel,
the constant for $\mathbb{S}^2$ and $\gamma=1$ was improved in \cite{I12JST}.
The same  bound for the two-dimensional torus (with equal periods)
was also obtained in \cite{I12JST}.

The case of multi-dimensional anisotropic torus $\mathbb{T}^d$ (a torus  with different
periods) was studied in~\cite{I-L-MS}. The natural condition
that the functions have zero mean value  over the whole torus was replaced
 by a stronger condition that  the functions have zero mean value over the shortest period
uniformly  for all remaining coordinates. Under this condition
 it was shown that independently of the ratios of the periods it holds
\begin{equation}\label{ILMS}
\mathrm{L}_{\gamma,d}(\mathbb{T}^d)\le \left(\frac\pi{\sqrt{3}}\right)^d
\cdot\mathrm{L}_{\gamma,n}^{\mathrm{cl}},\quad\gamma\ge1.
\end{equation}

Here we clearly have the Cartesian coordinates on $\mathbb{T}^d$, and,
furthermore,
one-dimensional  Lieb--Thirring estimates
for  the Scrh\"odinger operators   with matrix-valued potentials in
the periodic case~\cite{IMRN}, \cite{I-L-MS}. However, sharp semiclassical
Lieb--Thirring inequalities for $\gamma\ge3/2$ are
not known to hold for the periodic boundary conditions even
for scalar potentials, which leads, unlike~\eqref{DLL},  to the accumulation
along with the dimension of the factor
$\pi/\sqrt{3}$ in~\eqref{ILMS}.

As far as the sphere is concerned, none of the three points mentioned above is available.

In this work we consider Lieb--Thirring inequalities
in the dual formulation for orthonormal families on the two-dimensional sphere
$\mathbb{S}^2$ and the torus $\mathbb{T}^2$ (with equal periods)
by using the method of
 \cite{Rumin1}, \cite{Rumin2} (see also \cite{Frank}, \cite{lthbook}).
In the case of $\mathbb{R}^d$ this approach does not give the best
to date estimates of the constants obtained in \cite{D-L-L},
however, as we shall see, this  approach is very general and
flexible, and   works nicely in the case of the sphere producing
rather good constants in the Lieb--Thirring inequalities. We
restrict ourselves to the most important case of $\mathbb{S}^2$ and
consider the Lieb--Thirring inequalities in the following three
settings and prove in Section~\ref{sec2} that:
\begin{itemize}
  \item[] for an orthonormal family  $\{\psi_j\}_{j=1}^N\in
      H^1(\mathbb{S}^2)$ on the whole sphere with
      $\int_{\mathbb{S}^2}\psi_jdS=0$ it holds
\begin{equation}\label{item1}
\int_{\mathbb{S}^2}\biggl(\sum_{j=1}^N\psi_j(s)^2\biggr)^2dS\le\frac9{4\pi}
\sum_{j=1}^N\|\nabla\psi_j\|^2;
\end{equation}
\medskip
  \item[]
for an orthonormal family of  vector functions
$\{u_j\}_{j=1}^N\in H^1_0(\Omega,T\mathbb{S}^2)$
with supports in a domain $\Omega\subseteq\mathbb{S}^2$
\begin{equation}\label{item2}
\int_{\Omega}\biggl(\sum_{j=1}^N|u_j(s)|^2\biggr)^2dS\le\frac9{2\pi}
\sum_{j=1}^N\bigl(\|\div u_j\|^2+\|\rot u_j\|^2\bigr),
\end{equation}
where the constant goes over to $9/(4\pi)$ if $\div u_j=0$ (or $\rot u_j=0$);
\medskip
  \item[]
 for an orthonormal family $\{\psi_j\}_{j=1}^N\in H^1_0(\Omega)$,
$\Omega\subset\mathbb{S}^2$
\begin{equation}\label{item3}
\int_{\Omega}\biggl(\sum_{j=1}^N\psi_j(s)^2\biggr)^2dS\le\frac2{\pi}\,\frac{4\pi+|\Omega|}{4\pi-|\Omega|}
\sum_{j=1}^N\|\nabla\psi_j\|^2.
\end{equation}

\end{itemize}

We observe that the constants in \eqref{item1},\eqref{item2}
significantly improve the previously known estimate of the
constants~\cite{I12JST}, which was $3/2$. Also, the  constant in
\eqref{item3} blows up as $\Omega\to\mathbb{S}^2$,
since on the whole sphere the Lieb--Thirring inequality cannot hold
without the exclusion of the zero mode.
Sharp estimates for the first eigenvalue of
the Schr\"odinder operator on $\mathbb{S}^{d-1}$ when  the zero mode
is important
were obtained in \cite{D-E-L}.

In the same framework we consider in the end of Section~\ref{sec2}
the two-dimensional square torus $\mathbb{T}^2$ and show that for
an orthonormal family with orthogonality to constants it holds
\begin{equation}\label{itemtor}
\int_{\mathbb{T}^2}\biggl(\sum_{j=1}^N\psi_j(x)^2\biggr)^2dx\le\frac6{\pi^2}
\sum_{j=1}^N\|\nabla\psi_j\|^2.
\end{equation}

In Section~\ref{sec3}  we give applications of the obtained results
(namely, inequality~\eqref{item2})  to the Navier--Stokes system
in a domain $\Omega$ on $\mathbb{S}^2$ with Dirichlet boundary
conditions. The system reads:
\begin{equation}\label{NS}
\begin{cases}
\partial_t u+\nabla_uu+\nabla p=\nu \boldsymbol{\Delta} u+g,\\
u\big|_{t=0}=u^0,\ \ \div u=0,\ \ \ u\big|_{\partial\Omega}=0.
\end{cases}
\end{equation}
In the case $\Omega\subset\mathbb{R}^2$ we clearly have
$\nabla_v u=\sum_{i=1}^2v^i\partial_iu$, while on the sphere
$\nabla_vu$ is the covariant derivative of $u$
along $v$, and $\boldsymbol{\Delta}$ is the vector Laplace--de Rham
operator.

There is a vast literature devoted to the analysis of the
long time behaviour of the solutions of this system
described in terms of global attractors,
see, for instance, \cite{B-V}, \cite{T}
and the references therein. A global attractor $\mathscr A$ is a
compact strictly invariant set which attracts
bounded sets in the corresponding phase space as $t\to\infty$.
An important geometric characterization of the attractor
is its finite (Hausdorff and fractal) dimension.
For more than thirty years the estimate in \cite{T}
\begin{equation}\label{T85}
\dim_F\mathscr{A}\le c(\Omega)\frac{\|f\|\,|\Omega|}{\nu^2}\,,
\qquad\Omega\subset \mathbb{R}^2
\end{equation}
remains
the best available estimate in terms of the physical parameters:
the viscosity coefficient $\nu$, the size of the domain $|\Omega|=\mathrm{meas}(\Omega)$,
and the magnitude of the forcing term $\|f\|=\|f\|_{L_2(\Omega)}$.
We also point of the important paper \cite{Lieb} in this connection.
Finally, this estimate was written in an explicit form in~\cite{Il-Stokes}:
\begin{equation}\label{dimR2}
\dim_F\mathscr{A}\le\frac1{4\pi 3^{1/4}}\,\frac{\|f\||\Omega|}{\nu^2}=
0.060\ldots\cdot\frac{\|f\||\Omega|}{\nu^2}\,,
\qquad\Omega\subset \mathbb{R}^2.
\end{equation}

We show in Section~\ref{sec3} that in the case of $\Omega\subset\mathbb{S}^2$
the estimate of the type \eqref{T85} still holds,
and, furthermore, can also be written in an explicit form
\begin{equation}\label{dimS2}
\dim_F\mathscr{A}\le\frac3{2(2\pi)^{3/2}}\,\frac{\|f\||\Omega|}{\nu^2}=
0.095\ldots\cdot\frac{\|f\||\Omega|}{\nu^2}\,,
\qquad\Omega\subset \mathbb{S}^2.
\end{equation}
To prove \eqref{dimS2}, in addition to~\eqref{item2} we use the
following Li--Yau type  bound
\begin{equation}\label{LiYau}
\sum_{k=1}^n\lambda_k\ge\frac{2\pi}{|\Omega|}\,n^2
\end{equation}
for the eigenvalues $\lambda_k$ of the Stokes problem
on $\Omega\subset\mathbb{S}^2$:
\begin{equation}\label{Stokessmooth}
\aligned
&-\mathbf{\Delta} u_j \,+\, \nabla p_j\,=\,\lambda_ju_j,\\
&\div u_j\,=\,0,\,\,\,u_j\vert_{\partial\Omega }\,=\,0,
\endaligned
\end{equation}
which was  recently obtained in~\cite{I-L-IMRN} and is exactly
the same as in the case $\Omega\subset\mathbb{R}^2$ \cite{Il-Stokes}.

\setcounter{equation}{0}
\section{Lieb--Thirring inequalities on  $\mathbb{S}^2$ and $\mathbb{T}^2$}\label{sec2}

We first  recall the basic facts concerning the Laplace operator on
the sphere $\mathbb{S}^{2}$~(see, for instance,
\cite{S-W}).  We have for the (scalar) Laplace--Beltrami
operator $\Delta=\div\nabla$:
\begin{equation}\label{harmonics}
-\Delta Y_n^k=n(n+1) Y_n^k,\quad
k=1,\dots,2n+1,\quad n=0,1,2,\dots.
\end{equation}
Here the $Y_n^k$ are the orthonormal real-valued spherical
harmonics and each eigenvalue $\Lambda_n:=n(n+1)$ has multiplicity $2n+1$.

The following identity is essential in what
follows~\cite{Mikhlin}, \cite{S-W}: for any $s\in\mathbb{S}^{2}$
\begin{equation}\label{identity}
\sum_{k=1}^{2n+1}Y_n^k(s)^2=\frac{2n+1}{4\pi}.
\end{equation}
 This identity, in turn, follows from  the Laplace addition theorem for the
spherical harmonics on $\mathbb{S}^2$:
\begin{equation}\label{Lapl-add}
\sum_{k=1}^{2n+1}Y_n^k(s)Y_n^k(s_0)=\frac{2n+1}{4\pi}\,
P_n(s\cdot s_0),
\end{equation}
where $P_n(x)$, $x\in[-1,1]$ are the classical Legendre
polynomials.

\subsection*{The scalar case} We first consider the scalar case.
We exclude the zero mode and introduce the following notation
labelling the eigenfunctions and the corresponding eigenvalues with
a single subscript counting multiplicities: $-\Delta y_i=\lambda_i
y_i$, where
\begin{equation}\label{not}
\{y_i\}_{i=1}^\infty=\{
Y_n^1,\dots,Y_n^{2n+1}\}_{n=1}^\infty,\quad
\{\lambda_i\}_{i=1}^\infty=
\underset{\!2n+1\ \text{times}}{\{\Lambda_n,
\dots,\Lambda_n\}_{n=1}^{\infty}}.
\end{equation}

\begin{theorem}\label{Th:S2}
Let $\{\psi_j\}_{j=1}^N\in H^1(\mathbb{S}^2)$ be an orthonormal family of scalar functions
with zero average: $\int_{\mathbb{S}^2}\psi_j(s)dS=0$. Then $\rho(s):=\sum_{j=1}^N\psi_j(s)^2$
satisfies the inequality
\begin{equation}\label{LTS2}
\int_{\mathbb{S}^2}\rho(s)^2dS\le\frac9{4\pi}\sum_{j=1}^N\|\nabla\psi_j\|^2.
\end{equation}
\end{theorem}
\begin{proof} For $E\ge0$  we define the spectral
projections
$$
P_E=\sum_{\lambda_j<E}(\cdot,y_j)y_j\quad\text{and}\quad
P_E^\perp=\sum_{\lambda_j\ge E}(\cdot,y_j)y_j.
$$
We denote by $\dot\Delta$ the Laplace operator restricted to
the invariant subspace of functions orthogonal to constants, that is,
$\dot\Delta=P^\perp_2\Delta P^\perp_2$. Then
$$
-\dot\Delta=\sum_{j=1}^\infty\lambda_j(\cdot,y_j)y_j.
$$
Next, since
$$
\sum_{j=1}^\infty\lambda_ja_j=
\lambda_1(a_1+a_2+\dots)+(\lambda_2-\lambda_1)(a_2+a_3+\dots)+\dots=
\int_0^\infty\sum_{\lambda_j\ge E} a_jdE,
$$
it follows that
\begin{equation}\label{Eperp}
-\dot\Delta=\int_0^\infty\sum_{\lambda_j\ge E}(\cdot,y_j)y_jdE=\int_0^\infty P_E^\perp dE.
\end{equation}
Given the orthonormal family $\{\psi_j\}_{j=1}^N$, let $\varGamma$
be the finite rank orthogonal projection
$$
\varGamma=\sum_{j=1}^N(\cdot,\psi_j)\psi_j.
$$
Then $\Tr(-\dot\Delta)\varGamma=\Tr(-\Delta)\varGamma=\sum_{j=1}^N\|\nabla\psi_j\|^2$, which is on the right-hand
side in~\eqref{LTS2}. On the other hand, by the cyclic property of the trace and~\eqref{Eperp}
we have
\begin{equation}\label{intint}
\aligned
\sum_{j=1}^N\|\nabla\psi_j\|^2=
\Tr(-\dot\Delta)\varGamma=\Tr\varGamma(-\dot\Delta)\varGamma=\\=
\int_0^\infty\Tr\varGamma P_E^\perp\varGamma dE=
\int_0^\infty\Tr P_E^\perp\varGamma P_E^\perp dE=\\=
\int_0^\infty\int_{\mathbb{S}^2}\rho_{P_E^\perp\varGamma P_E^\perp}(s)dSdE=
\int_{\mathbb{S}^2}\int_0^\infty\rho_{P_E^\perp\varGamma P_E^\perp}(s)dE dS,
\endaligned
\end{equation}
where
$$
\rho_{P_E^\perp\varGamma P_E^\perp}(s)=\sum_{j=1}^N\|P_E^\perp\psi_j\|^2\psi_j(s)^2.
$$
Now let $B$ be  a spherical cap around a point $a\in\mathbb{S}^2$
with area $|B|$ and let
$\chi_B$ be the corresponding characteristic function. Also, let $E\ge2\,(=\lambda_1=\Lambda_1)$. Then denoting by
$\|\cdot\|_{\mathrm{HS}}$ the Hilbert--Schmidt norm we obtain
\begin{equation}\label{|B|}
\aligned
\left(\int_B\rho(s)dS\right)^{1/2}=\|\varGamma\,\chi_B\|_{\mathrm{HS}}\le
\|\varGamma P_E\chi_B\|_{\mathrm{HS}}+\|\varGamma P_E^\perp \chi_B\|_{\mathrm{HS}}=\\=
\|\varGamma P_E\chi_B\|_{\mathrm{HS}}+\left(\int_B\rho_{P_E^\perp\varGamma P_E^\perp}(s) dS\right)^{1/2}.
\endaligned
\end{equation}
Using that $\|\varGamma\|=1$ and then the fact that both $\chi_B$ and $P_E$ are projections,
we find that
\begin{equation}\label{2nplus1}
\aligned
\|\varGamma P_E\chi_B\|_{\mathrm{HS}}^2\le\| P_E\chi_B\|_{\mathrm{HS}}^2=
\sum_{\lambda_j<E}\int_{\mathbb{S}^2}y_j(s)^2\chi_B(s)dS=\\=
\sum_{n=1}^{n(n+1)<E}\int_{\mathbb{S}^2}\sum_{k=1}^{2n+1}Y_n^k(s)^2\chi_B(s)dS=
\frac1{4\pi}|B|\sum_{n=1}^{n(n+1)<E}(2n+1),
\endaligned
\end{equation}
where we used the key identity~\eqref{identity} at the last step.

Let $n(E)=(\sqrt{1+4E}-1)/2$ be the positive root of the quadratic equation $n(n+1)=E$
and let $[n(E)]$ be the integer part of it. We estimate the sum on the right-hand side
$$
\aligned
\sum_{n=1}^{n(n+1)<E}(2n+1)=([n(E)]+1)^2-1\le (n(E)+1)^2-1=\\=
\frac{4E+2\sqrt{1+4E}-2}4\le E\max_{E\ge2}\frac{4E+2\sqrt{1+4E}-2}{4E}=\frac32E.
\endaligned
$$
Substituting this into \eqref{|B|} and letting $|B|\to0$ we obtain
$$
\rho(a)^{1/2}\le C^{1/2}E^{1/2}+\rho_{P_E^\perp\varGamma P_E^\perp}(a)^{1/2},\quad
C=\frac3{8\pi}\,.
$$
Since $\rho_{P_E^\perp\varGamma P_E^\perp}\ge0$ we have actually shown that
\begin{equation}\label{estwithplus}
\rho_{P_E^\perp\varGamma P_E^\perp}(a)\ge \bigl(\rho(a)^{1/2}-C^{1/2}E^{1/2}\bigl)^2_+.
\end{equation}
We derived this lower bound under the condition that $E\ge2$. However, since
$P_E^\perp=I$ for $0\le E <2$ and $\rho_{P_E^\perp\varGamma P_E^\perp}=\rho$,
lower bound \eqref{estwithplus} holds for all $E\ge0$. Integrating with respect to $E$ we obtain
$$
\int_0^\infty\rho_{P_E^\perp\varGamma P_E^\perp}(a)dE
\ge \int_0^\infty \bigl(\rho(a)^{1/2}-C^{1/2}E^{1/2}\bigl)^2_+dE=
\frac1{6C}\rho(a)^2=\frac{4\pi}{9}\rho(a)^2.
$$
Finally, integration over $\mathbb{S}^2$ in \eqref{intint} completes the proof:
$$
\sum_{j=1}^N\|\nabla\psi_j\|^2\ge \frac{4\pi}{9} \int_{\mathbb{S}^2}\rho(s)^2dS.
$$
\end{proof}

\subsection*{The vector case}

In the vector case we have the identity for the gradients of
spherical harmonics that is similar to~\eqref{identity} (see
\cite{I93}): for any $s\in\mathbb{S}^{2}$
\begin{equation}\label{identity-vec}
\sum_{k=1}^{2n+1}|\nabla Y_n^k(s)|^2=n(n+1)\frac{2n+1}{4\pi}.
\end{equation}

This identity is essential for inequalities for vector
functions on $\mathbb{S}^2$. Substituting $\varphi(s)=Y_n^k(s)$ into the identity
$$
\Delta\varphi^2=2\varphi\Delta\varphi+2|\nabla\varphi|^2
$$
we sum the results over $k=1,\dots,2n+1$. In view of
\eqref{identity} the left-hand side vanishes and we
obtain~\eqref{identity-vec} since the $Y_n^k(s)$'s are the
eigenfunctions corresponding to $\Lambda_n$.

In the vector case we first define the Laplace operator
acting on (tangent) vector
fields on $\mathbb{S}^2$ as the Laplace--de Rham
operator $-d\delta-\delta d$ identifying $1$-forms and
vectors. Then for a two-dimensional manifold
(not necessarily $\mathbb{S}^2$) we have
\cite{I90, I93}
\begin{equation}\label{vecLap}
\mathbf{\Delta} u=\nabla\div u-\rot\rot u,
\end{equation}
where the operators $\nabla=\grad$ and $\div$ have the
conventional meaning. The operator $\rot$ of a vector $u$ is a
scalar  and for a scalar $\psi$,
$\rot\psi$ is a vector:
$$
\rot u:=-\div(n\times u),\qquad
\rot\psi:=-n\times\nabla\psi,
$$
where $n$ is the unit outward normal vector, and in the local
frame $n\times u=(-u_2,u_1)$.

 Integrating by parts
we obtain
\begin{equation}\label{byparts}
(-\mathbf{\Delta} u,u)_{L_2(T\mathbb{S}^2)}=\|\rot u\|^2+\|\div u\|^2.
\end{equation}

The vector Laplacian has a complete in $L_2(T\mathbb{S}^2)$ orthonormal basis
of vector eigenfunctions. Using notation \eqref{not}
we have
\begin{equation}\label{bas-vec}
\aligned
-\mathbf{\Delta} w_j=\lambda_j w_j,\qquad -\mathbf{\Delta}
v_j=\lambda_j v_j,
\endaligned
\end{equation}
where
$$
 w_j=\lambda_j^{-1/2}n\times\nabla y_j,\ \ \div w_j=0,\qquad
v_j=\lambda_j^{-1/2}\nabla y_j,\ \ \rot v_j=0.
$$

Hence,  on $\mathbb{S}^2$, corresponding
to the eigenvalue
$\Lambda_n=n(n+1)$, where $n=1,2,\dots$, there are two families of $2n+1$
orthonormal vector-valued  eigenfunctions  $w_n^k(s)$ and $v_n^k(s)$, where
$k=1,\dots,2n+1$ and~(\ref{identity-vec}) gives the
following important identities: for any $s\in\mathbb{S}^2$
\begin{equation}\label{id-vec}
\sum_{k=1}^{2n+1}|w_n^k(s)|^2=\frac{2n+1}{4\pi},\qquad
\sum_{k=1}^{2n+1}|v_n^k(s)|^2=\frac{2n+1}{4\pi}.
\end{equation}
We finally observe that since the sphere is simply connected,  $-\mathbf{\Delta}$ is strictly
positive $-\mathbf{\Delta}\ge \Lambda_1I=2I.$

\begin{theorem}\label{Th:LT-vec}
Let $\{u_j\}_{j=1}^N\in H^1(T\mathbb{S}^2)$
be an  orthonormal family  of vector fields in $L^2(T\mathbb{S}^2)$.  Then
\begin{equation}\label{orthvec}
\int_{\mathbb{S}^2}\rho(s)^2dS\le
 \frac9{2\pi}\sum_{j=1}^N(\|\rot u_j\|^2
+\|\div u_j\|^2),
\end{equation}
where $\rho(s)=\sum_{j=1}^N|u_j(s)|^2$.
If, in addition, $\div u_j=0$ $($or $\rot u_j=0$$)$,
then
\begin{equation}\label{orthvecsol}
\int_{\mathbb{S}^2}\rho(s)^2dS\le\frac9{4\pi}\cdot
\begin{cases}\displaystyle
\sum_{j=1}^N\|\rot u_j\|^2,
\quad \ \div u_j=0,
\\\displaystyle
\sum_{j=1}^N\|\div u_j\|^2,
\quad \rot u_j=0.
\end{cases}
\end{equation}
\end{theorem}
\begin{proof} We prove~\eqref{orthvec}, the proof of \eqref{orthvecsol}
is  similar.
For $E\ge0$  we  set
$$
P_E=\sum_{\lambda_j<E}\bigl((\cdot,w_j)w_j+(\cdot,v_j)v_j\bigr),\quad
P_E^\perp=\sum_{\lambda_j\ge E}\bigl((\cdot,w_j)w_j+(\cdot,v_j)v_j\bigr).
$$
Then as before
$$
-\mathbf{\Delta}=\int_0^\infty\sum_{\lambda_j\ge E}\bigl((\cdot,w_j)w_j+(\cdot,v_j)v_j\bigr)dE=\int_0^\infty P_E^\perp dE.
$$
The finite rank operator  $\varGamma$ is now
$
\varGamma=\sum_{j=1}^N(\cdot,u_j)u_j,
$
and in view of \eqref{byparts} equality~\eqref{intint} goes over to
$$
\aligned
\sum_{j=1}^N\bigl(\|\div u_j\|^2+\|\rot u_j\|^2\bigr)=
\Tr(-\mathbf{\Delta})\varGamma=
\int_{\mathbb{S}^2}\int_0^\infty\rho_{P_E^\perp\varGamma P_E^\perp}(s)dE dS,
\endaligned
$$
where
$$
\rho_{P_E^\perp\varGamma P_E^\perp}(s)=\sum_{j=1}^N\|P_E^\perp u_j\|^2|u_j(s)|^2.
$$
Next, \eqref{|B|} is formally unchanged and \eqref{2nplus1} in view of~\eqref{id-vec}
goes over to
$$
\aligned
\| P_E\chi_B\|_{\mathrm{HS}}^2=\!
\sum_{\lambda_j<E}\int_{\mathbb{S}^2}\bigl(|w_j(s)|^2+|v_j(s)|^2\bigr)\chi_B(s)dS=
\frac2{4\pi}|B|\!\sum_{n=1}^{n(n+1)<E}\!(2n+1),
\endaligned
$$
and we complete the proof as in Theorem~\ref{Th:S2} with $C=3/(4\pi)$.
\end{proof}

\subsection*{Lieb--Thirring inequality on domains on $\mathbb{S}^2$}
In conclusion we consider Lieb--Thirring inequalities for scalar functions  in a domain
on $\mathbb{S}^2$ with Dirichlet boundary conditions. Let $\{\psi_j\}_{j=1}^N\in H^1_0(\Omega)$
be an orthonormal family in $L_2(\Omega)$,
where $\Omega\Subset\mathbb{S}^2$ is a proper domain on $\mathbb{S}^2$ with
$|\Omega|<4\pi$. Of course, the Lieb--Thirring inequality holds for this system,
and we find below an estimate for the constant.

We extend the $\psi_j$'s by zero to the whole sphere (preserving the same notation).
The functions $\psi_j$, however, do not necessarily have zero average. Therefore we
consider the operator
$$
A=-\Delta+aI,\qquad a>0.
$$
on the whole space $L_2(\mathbb{S}^2)$ (without the orthogonality condition).
Then $\mathrm{infspec}A=a$. We also include the eigenvalue $\lambda_0=\Lambda_0=0$
into~\eqref{not}, so that it goes over to
\begin{equation}\label{not0}
\{y_i\}_{i=0}^\infty=\{
Y_n^1,\dots,Y_n^{2n+1}\}_{n=0}^\infty,\quad
\{\lambda_i\}_{i=0}^\infty=
\underset{\!2n+1\ \text{times}}{\{\Lambda_n,
\dots,\Lambda_n\}_{n=0}^{\infty}}.
\end{equation}
Following the proof of Theorem~\ref{Th:S2} for $E\ge0$   the projections are now
$$
P_E=\sum_{\lambda_j+a<E}(\cdot,y_j)y_j\quad\text{and}\quad
P_E^\perp=\sum_{\lambda_j+a\ge E}(\cdot,y_j)y_j.
$$
Then
$$
A=\sum_{j=0}^\infty(\lambda_j+a)(\cdot,y_j)y_j,
$$
and
\begin{equation}\label{Eperp1}
A=\int_0^\infty\sum_{\lambda_j+a\ge E}(\cdot,y_j)y_jdE=\int_0^\infty P_E^\perp dE.
\end{equation}
As in Theorem~\ref{Th:S2} the operator $\varGamma=\sum_{j=1}^N(\cdot,\psi_j)\psi_j$
and~\eqref{intint} goes over to
$$
\aligned
\sum_{j=1}^N\bigl(\|\nabla\psi_j\|^2+a\|\psi_j\|^2\bigr)=
\Tr A\varGamma=
\int_{\mathbb{S}^2}\int_0^\infty\rho_{P_E^\perp\varGamma P_E^\perp}(s)dE dS.
\endaligned
$$
Inequality \eqref{|B|} is formally as before  and \eqref{2nplus1} becomes
$$
\aligned
\| P_E\chi_B\|_{\mathrm{HS}}^2=
\sum_{\lambda_j+a<E}\int_{\mathbb{S}^2}y_j(s)^2\chi_B(s)dS=
\frac1{4\pi}|B|\!\sum_{n=0}^{n(n+1)+a<E}\!(2n+1)\le\\\le
\frac1{4\pi}|B|\bigl(N(E,a)+1\bigr)^2=
\frac1{4\pi}|B|\frac{\bigl(\sqrt{1+4(E-a)}+1\bigr)^2}4\le
\frac1{4\pi}|B|E\,F(a),
\endaligned
$$
where $N(E,a)$ is the root of the equation $n(n+1)-(E-a)=0$ and where
\begin{equation}\label{F(a)}
F(a):=\max_{E\ge a}\frac{\bigl(\sqrt{1+4(E-a)}+1\bigr)^2}{4E}=
\left\{
  \begin{array}{ll}
    1/a, & \hbox{$a\le 1/2$;} \\
    {4a}/{(4a-1)},\ & \hbox{$a\ge 1/2$.}
  \end{array}
\right.
\end{equation}
We can now complete the argument as in Theorem~\ref{Th:S2} to obtain the
inequality
\begin{equation}\label{mina}
\aligned
\int_{\Omega}\rho(s)^2dS\le\frac{3F(a)}{2\pi}\sum_{j=1}^N\bigl(\|\nabla\psi_j\|^2+a\|\psi_j\|^2\bigr)
\le\\\le
\frac{3F(a)}{2\pi}\bigl(1+a\lambda_1(\Omega)^{-1}\bigr)\sum_{j=1}^N\|\nabla\psi_j\|^2,
\endaligned
\end{equation}
where $\lambda_1(\Omega)$ is the first eigenvalue of the Dirichlet Laplacian in
$\Omega$, for which we have the following lowed bound obtained in \cite[Corollary~3.2]{I-L-IMRN}:
$$
\lambda_1(\Omega)\ge\frac{2\pi}{|\Omega|}\,\left(1-\frac{|\Omega|}{4\pi}\right).
$$
Setting $a=1$ (so that $F(a)=4/3$)
we have proved the following result.
\begin{theorem}\label{Th:S2Omega}
Let $\Omega$ be a domain on $\mathbb{S}^2$ with $|\Omega|<|\mathbb{S}^2|=4\pi$.
Let $\{\psi_j\}_{j=1}^N\in H^1_0(\Omega)$ be an orthonormal family of scalar functions. Then $\rho(s):=\sum_{j=1}^N\psi_j(s)^2$
satisfies the inequality
\begin{equation}\label{LTS2Omega}
\int_{\Omega}\rho(s)^2dS\le k(\Omega)
\sum_{j=1}^N\|\nabla\psi_j\|^2_{L^2(\Omega)}, \quad k(\Omega)\le\frac2{\pi}\,\frac{4\pi+|\Omega|}{4\pi-|\Omega|}.
\end{equation}
\end{theorem}

\begin{remark}
{\rm
The constant in~\eqref{LTS2Omega} blows up as $|\Omega|\to 4\pi$ as it should,
since the Lieb--Thirring inequality for scalar functions cannot hold on the whole sphere
without the exclusion on the zero mode. This does not happen in the vector case,
since the vector Laplacian is positive definite, and for any
$\Omega\subseteq\mathbb{S} ^2$
and an orthonormal family $\{u_j\}_{j=1}^N\in H^1_0(\Omega,T\mathbb{S}^2)$
extension by zero shows that
the corresponding Lieb--Thirring
constants are uniformly bounded by the constant on the whole sphere.
}
\end{remark}
\begin{remark}
{\rm
In~\eqref{LTS2Omega} we have just set $a=1$, while of course
$$
k(\Omega)\le \frac3{2\pi}\min_{a\ge 1/2}
F(a)\bigl(1+a\lambda_1(\Omega)^{-1}\bigr).
$$
Given~\eqref{F(a)} we can find the minimum explicitly.
Avoiding writing a  long expression we just write
the  asymptotic behaviour of our estimate of the constant  in the two regimes:
$$
k(\Omega)\to \frac3{2\pi}\ \text{as}\ |\Omega|\to0;\ \
k(\Omega)\sim \frac{12}{4\pi-|\Omega|}\ \text{as}\ |\Omega|\to4\pi.
$$
However, setting $a=1/2$ in~\eqref{mina} we see that
$k(\Omega)\le 12/(4\pi-|\Omega|)$ for all $|\Omega|<4\pi$.
It is also worth pointing out that $3/{(2\pi)}$ is the estimate of the
Lieb--Thirring constant in $\mathbb{R}^2$ by the approach of
\cite{Rumin1}, \cite{Rumin2},  (see also \cite{lthbook}).
}
\end{remark}

\subsection*{Lieb--Thirring inequality on the torus}
The case of the square 2D torus $\mathbb{T}^2=[0,2\pi]^2$
is treated in exactly the same way as in Theorem~\ref{Th:S2}. The orthonormal eigenfunctions
of $-\Delta$ in the space of functions with zero average are
$$
\frac1{2\pi}e^{ik\cdot x},\quad k\in\mathbb{Z}^2_0=\{k=(k_1,k_2)\in\mathbb{Z}^2, |k|^2>0\}
$$
with eigenvalues
$k_1^2+k_2^2$. The first eigenvalue is $1$ and for $E\ge1$ the number $N(E)$ of eigenvalues
less than or equal to $E$ satisfies
\begin{equation}\label{NE}
N(E)\le 4E.
\end{equation}
In fact, it is well known that $\lim_{E\to\infty}N(E)/E=\pi$.
Therefore it suffices to verify~\eqref{NE} for finitely many
eigenvalues. We have $N(1)=4$, $N(5)=20$ and for other $E$ we have
the strict inequality in~\eqref{NE}.

\begin{theorem}\label{Th:T2}
Let $\{\psi_j\}_{j=1}^N\in H^1(\mathbb{T}^2)$ be an orthonormal family of scalar functions
with zero average: $\int_{\mathbb{T}^2}\psi_j(x)dx=0$. Then $\rho(x):=\sum_{j=1}^N\psi_j(x)^2$
satisfies the inequality
\begin{equation}\label{LTT2}
\int_{\mathbb{T}^2}\rho(x)^2dx\le\frac6{\pi^2}\sum_{j=1}^N\|\nabla\psi_j\|^2.
\end{equation}
\end{theorem}
\begin{proof}
Up to \eqref{|B|} the proof is the same as in Theorem~\ref{Th:S2}.
Setting now $E\ge1$, in view of~\eqref{NE} we have
$$
\aligned
\| P_E\chi_B\|_{\mathrm{HS}}^2=\frac1{4\pi^2}
\sum_{k\in\mathbb{Z}_0^2,\, |k|^2<E}\int_{\mathbb{T}^2}
\chi_B(x)dx\le\frac1{4\pi^2}N(E)|B|\le\frac1{\pi^2}E|B|,
\endaligned
$$
and we obtain the analog of~\eqref{estwithplus} with $C=1/\pi^2$.
The completion of the proof is again exactly the same.
\end{proof}

\begin{remark}\label{R:T2}
{\rm
The case of multidimensional anisotropic torus with arbitrary
periods was studied in~\cite{I-L-MS} under a stronger orthogonality
condition that the functions have  zero average with respect to the
shortest coordinate. Then, for example, in the 2D case the constant
is $\pi/6<6/\pi^2$ and is independent of the aspect ratio.
}
\end{remark}

\begin{remark}\label{R:Lad-ineq}
{\rm
It is also worth pointing out that the 2D Lieb--Thirring
inequalities for one function, that is, for $N=1$ turn into
multiplicative inequality of the form
$$
\|\psi\|_{L^4}^4\le k\|\psi\|^2\|\nabla\psi\|^2,
$$
where $k$ is the corresponding Lieb--Thirring constant. In the
context of the Navier--Stokes equations this inequality is called
the Ladyzhenskaya inequality and is important both in the theory,
and applications.  For example, the  estimate for $k$ specifically
for this inequality on the torus $\mathbb{T}^2$ in~\cite{k-Foias}
is $1/(2\pi^2)+1/({\sqrt{2}\pi})+2=2.27\dots$, while
 \eqref{itemtor}  gives $6/\pi^2=0.60\dots$.
 }
\end{remark}

\setcounter{equation}{0}
\section{  Navier--Stokes system  in a domain on $\mathbb{S}^2$}\label{sec3}

We now consider the  Navier--Stokes system in a domain $\Omega\subset\mathbb{S}^2$,
see~\eqref{NS} and write it as an evolution
equation
\begin{equation}\label{NSE}
\partial_t u+\nu A u+B(u,u)=f,\qquad u(0)=u_0,
\end{equation}
in the phase space $H$, which is the completion in $L_2(\Omega,T\mathbb{S}^2)$
of the linear set of smooth divergence free tangent vector
functions compactly supported in $\Omega$. Let $P$ be the corresponding
orthogonal projection.
Then $A=-P\boldsymbol{\Delta}$ is the Stokes operator and
$B(u,v)=P(\nabla_uv)$ is the bilinear term defined by duality as follows:
for all $u,v,w\in H^1_0(\Omega,T\mathbb{S}^2)\cap H$
$$
\aligned
&\langle Au,v\rangle=(\rot u,\rot v),\\
&\langle B(u,v),w\rangle=\int_{\Omega}
\nabla_{u(s)}v(s)\cdot w(s)\,dS=:b(u,v,w).
\endaligned
$$

Equation (\ref{NSE})
generates the semigroup of solution operators $S_t\colon H\to H$, $S_tu_0=u(t)$
which  has a compact global attractor
$\mathcal{A}\Subset H$. The case of whole sphere
(or a two-dimensional manifold without boundary)
was studied in \cite{I90}, \cite{I93}. The fact that
$\Omega$ is not the whole sphere does not play a role here
(see, however, Remark~\ref{R:wholeS2}). The
existence of the attractor is shown similarly to the
case of a bounded domain in $\mathbb{R}^2$ with Dirichlet boundary conditions
  (see, for instance, \cite{B-V},
 \cite{T} for the case
of a domain with  smooth boundary  and \cite{LadNon},
\cite{Rosa} for a nonsmooth domain). The attractor
$\mathcal{A}$ is the maximal strictly invariant compact  set.

\begin{theorem}\label{T:dim}
The fractal dimension of  $\mathscr{A}$ satisfies the estimate
\begin{equation}\label{dim}
\dim_F\mathscr{A}\le \frac3{2(2\pi)^{3/2}}\,\frac{\|f\||\Omega|}{\nu^2}\,.
\end{equation}
\end{theorem}

\begin{proof}
Following the  general scheme
we have to estimate the $m$-trace of the semigroup
linearized  on the solution lying on the attractor
\cite{T}.
The semigroup $S_t$ is uniformly
 differentiable in $H$
with differential $L(t,u_0)\colon\xi\to U(t)\in~H$,
where $U(t)$ is the solution of the variational equation
\begin{equation}\label{var-eq}
\partial_tU=-\nu AU- B(U,u(t))-B(u(t),U)=:{\mathcal L}(t,u_0)U, \qquad U(0)=\xi.
\end{equation}
We estimate the numbers $q(m)$ (the sums of the first $m$
global Lyapunov exponents):
\begin{equation}\label{def-q(m)}
q(m):=\limsup_{t\to\infty}\,\sup_{u_0\in {\mathcal A}}\,\sup_{\{v_j\}_{j=1}^m} \frac{1}t
\int_0^t \sum_{j=1}^m({\mathcal L}(\tau,u_0)v_j,v_j)\,d\tau,
\end{equation}
where $\{v_j\}_{j=1}^m\in H^1_0(\Omega,T\mathbb{S}^2)\cap H$ is an arbitrary
orthonormal system
(see \cite{B-V}, \cite{C-F85},  \cite{T}). Using below the well-known identity
$b(u,v,v)=0$ and the estimate
$$
\sum_{j=1}^m b(v_j,u,v_j)\le
2^{-1/2}\|\rho\|\|\rot u\|,
$$
where
$\rho(s)=\sum_{j=1}^m|v_j(s)|^2$,
see \cite[Lemma~3.2]{IMT}, we obtain also using~\eqref{byparts}
\begingroup
\allowdisplaybreaks
\begin{align*}
&\sum_{j=1}^m({\mathcal L} (t,u_0)v_j,v_j) = -\nu\sum_{j=1}^m\|\rot v_j\|^2-
\sum_{j=1}^m b(v_j,u(t),v_j)\le\\
&\qquad
\le-\nu\sum_{j=1}^m\|\rot v_j\|^2+
2^{-1/2}\|\rho\|\|\rot u(t)\|\le\\
&\qquad\le-\nu\sum_{j=1}^m\|\rot v_j\|^2+ 2^{-1/2}
\biggl(c_{\mathrm{LT}} \sum_{j=1}^m\|\rot v_j\|^2\biggr)^{1/2}\|\rot u(t)\|\le\\
&\qquad\le-\frac\nu2\sum_{j=1}^m\|\rot v_j\|^2+ \frac{c_\mathrm{LT}}{4\nu}\|\rot u(t)\|^2 \le
-\frac{\nu c_\mathrm{LY}m^2}{2|\Omega|}+ \frac{c_\mathrm{LT}}{4\nu}\|\rot u(t)\|^2.
\end{align*}
\endgroup
Here we used the Lieb--Thirring inequality~\eqref{orthvecsol} with
$c_{\mathrm{LT}}\le 9/(4\pi)$ and inequality
\eqref{LiYau}, which gives by the variational principle that
for  the orthonormal family $\{v_j\}_{j=1}^m\in H^1_0(\Omega,T\mathbb{S}^2)\cap H$
 it holds
\begin{equation}\label{n=2}
\sum_{k=1}^m\|\rot v_k\|^2\ge \sum_{k=1}^m\lambda_k\ge
\frac {c_{\mathrm{LY}}m^2}{|\Omega|},\qquad c_{\mathrm{LY}}=2\pi,
\end{equation}
where $\lambda_k$ are the eigenvalues of the Stokes operator~\eqref{Stokessmooth}.
Using the well-known estimate~\cite{B-V}, \cite{T}
$$
\limsup_{t\to\infty}\sup_{u_0\in\mathscr{A}} \frac 1t\int_0^t\|\rot u(\tau)\|^2d\tau
\le\frac{\|f\|^2}{\lambda_1\nu^2},
$$
for the solutions lying on the attractor we
obtain  for the numbers $q(m)$:
$$
q(m)\le -\frac{\nu c_{\mathrm{LY}}m^2}{2|\Omega|}+ \frac{c_{\mathrm{LT}}\|f\|^2}{4\lambda_1\nu^3}
\le-\frac{\nu c_{\mathrm{LY}}m^2}{2|\Omega|}+
\frac{c_{\mathrm{LT}}\|f\|^2|\Omega|}{4c_{\mathrm{LY}}\nu^3}\,.
$$
It was shown in~\cite{C-F85}, \cite{T}
and in~\cite{Ch-I2001}, \cite{Ch-I}, respectively,
that both the Hausdorff and fractal dimensions of $\mathscr{A}$
are bounded by the number $m_*$ for which
$q(m_*)=0$.
This gives that
$$
\dim_F\mathscr{A}\le \frac1{c_{\mathrm{LY}}}\biggl(\frac{c_{\mathrm{LT}}}{2}\biggr)^{1/2}
\frac{\|f\||\Omega|}{\nu^2}\,
$$
and completes the proof since $c_{\mathrm{LY}}=2\pi$ and $c_{\mathrm{LT}}\le9/(4\pi)$.
\end{proof}
\begin{remark}\label{R:R2S2}
{\rm
In terms of the physical parameters estimate~\eqref{dim}
is the same as \eqref{dimR2} in the case $\Omega\subset\mathbb{R}^2$,
so that the (constant) curvature does not seem to play a role. Furthermore,
the numerical coefficients are different only due to the fact that
the current estimate for the Lieb--Thirring constant for
orthonormal families of two-dimensional divergence free vector
functions in $\Omega\subset\mathbb{R}^2$ is $1/(2\sqrt{3})$ (see~\cite{D-L-L}, \cite{Il-Stokes}),
which is  better than our estimate for the sphere $9/(4\pi)$,
while, as mentioned above,  the constants in the  Li-Yau lower bounds for the Stokes operator
in $\Omega\subset\mathbb{R}^2$ and $\Omega\subset\mathbb{S}^2$
are the same.
}
\end{remark}

\begin{remark}\label{R:wholeS2}
{\rm
In the case when $\Omega=\mathbb{S}^2$
(or, more generally, when $\Omega$ is a 2D
manifold without boundary)  we have an additional
important orthogonality relation
$$
b(u,u,Au)=0,
$$
which is the same as in the 2D space periodic case. This makes it possible to improve the estimate
of the dimension in terms of the physical parameters and obtain that
as in the case of the torus $\mathbb{T}^2$ \cite{T} it holds (see~\cite{I93, I94})
$$
\dim_F\mathscr{A}\le c\, G^{2/3}(1+\log G)^{1/3}, \qquad G=\frac{\|f\||\Omega|}{\nu^2}\,.
$$
}
\end{remark}

\subsection*{Acknowledgements}\label{SS:Acknow}
A.I.  acknowledges the warm hospitality of the Mit\-tag--Leffler
Institute where this work has started. The work of A.I.
 was supported by  Programme no.1 of the Russian Academy of Sciences.
 A.L. was partially
funded  by the grant of the Russian Federation Government
to support research under the supervision of a leading
scientist at the Siberian Federal University, 14.Y26.31.0006.

\end{document}